\documentclass{amsart}
\usepackage{amsmath,amssymb}
\usepackage{graphicx,subfigure}
\usepackage{mathtools,slashed}
\newtheorem{theorem}{Theorem}[section]
\newtheorem{proposition}[theorem]{Proposition}

\newtheorem{lemma}[theorem]{Lemma}

\theoremstyle{definition}
\newtheorem{definition}[theorem]{Definition}

\theoremstyle{remark}

\numberwithin{equation}{section}

\def\NN{\mathbb{N}}

\def\RR{\mathbb{R}}

\def\TT{\mathbb{T}}

\renewcommand\SS{\mathbb{S}}
\def\QQ{\mathbb{Q}}
\newcommand\minus\backslash

\newcommand\lan\langle
\newcommand\ran\rangle

\DeclareMathOperator\Div{div}

\renewcommand\leq\leqslant
\renewcommand\geq\geqslant
\newlength{\intwidth}

 \DeclareMathOperator\curl{curl}

\begin{document}

\title{Periodic orbits of analytic Euler fields on 3-manifolds}

\author{Francisco Torres de Lizaur}
\address{Max Planck Institute for Mathematics, Vivatsgasse 7, 53111 Bonn, Germany}
\email{ftorresdelizaur@mpim-bonn.mpg.de}

\begin{abstract}

On any closed Riemannian 3-manifold which is not a torus bundle, every nonvanishing analytic solution of the stationary Euler equations has a periodic trajectory. This result is originally due to A. Rechtman \cite{RE} and K. Cieliebak and E. Volkov \cite{CV}; here we present an alternative proof of it.

\end{abstract}

\maketitle

\section{Introduction}
The motion of an incompressible inviscid fluid in a Riemannian 3-manifold $(M, g)$ is described by a time-dependent vector field $X$ satisfying the \emph{Euler equations}
\[
\frac{\partial X}{\partial t}+\nabla^{LC}_{X}X=-\nabla P ,\,\,\,\, \Div X=0 \,.
\]

Here $\Div$, $\nabla$, and $\nabla^{LC}$ denote the divergence, gradient and covariant derivative on $M$, defined with the metric $g$; and $P$ is
a (time-dependent) function called the pressure, which is also an unknown in the equations.

Solutions to the above equations that do not depend on time are called stationary, or steady, Euler flows. They describe the equilibrium configurations of the fluid. It is straightforward to see that in this case the Euler equations are equivalent to
\begin{equation*}
X\times \curl X=\nabla B\,, \;\; \Div X=0\,,
\end{equation*}
where $B:=P+\frac12|X|^2$ is called the Bernoulli function.

This paper is concerned with the qualitative properties of steady Euler flows and, more precisely, with the existence of closed trajectories amongst the flow lines.

A lot of aspects of steady Euler flows are not yet understood; nonetheless, the study of some of their qualitative properties has proven to be fruitful, revealing many connections with other areas of mathematics (see e.g  \cite{Kh} and \cite{P} for reviews). Arnold's structure theorem~\cite{AKh} is the first major result in this regard. Arnold discovered that, when the Bernoulli $B$ function is analytic, the flow lines outside its critical level sets  are arranged in a very simple manner: as 1-dimensional families of 2-dimensional integrable Hamiltonian systems. 

Completely opposite to the situation covered by Arnold's theorem, we have those steady Euler flows whose Bernoulli function has all level sets critical, i.e those having $ B\equiv $ constant. These solutions are often called Beltrami fields. 

Arnold conjectured that Beltrami fields should have very complex flow line arrangements, and since then they have attracted a lot of attention \cite{H, DH, EP, EPT}. Interestingly, contact topology has turned out to be a very useful tool for their study, at least in the absence of zeroes in the vector fields. The reason is that nonvanishing Beltrami fields are actually Reeb flows, and vice versa. More precisely, nonvanishing Beltrami fields satisfy the relation $\curl X=f X$ for some well-defined real-valued function $f$; when this function is constant, they are equivalent (after suitable rescaling) to Reeb flows of contact structures, an idea put forward by Sullivan and developed by Etnyre and Ghrist \cite{EG1}. When $f$ is not constant, the correspondence is rather with Reeb flows of stable Hamiltonian structures (also known as volume-preserving geodesible flows), as shown by Rechtman~\cite{RE2,RE}.

It is by means of this correspondence that we know that, on any closed (i.e compact without boundary) 3-manifold which is not a torus bundle, a $C^2$ nonvanishing Beltrami field always has periodic orbits (in a torus bundle it is straightforward to find examples of Beltrami fields without periodic orbits). Indeed this statement, expressed in terms of Reeb flows of stable Hamiltonian structures, was a major recent breakthrough in contact topology due to Hutchings and Taubes \cite{HT}, and independently to Rechtman \cite{RE} in the case of the 3-sphere and in manifolds with non-trivial second homotopy group.

The question of whether, on closed 3-manifolds which are not torus bundles, a general nonvanishing steady Euler field  has a periodic orbit, remains completely open. An a priori promising approach to this problem would be to show that any nonvanishing steady Euler field $X$ is \emph{beltramisable}, i.e that it satisfies $\curl X= f X$ for \emph{some metric} (note that the curl operator depends on the metric). However, this is known to be false in general: Cieliebak and Volkov provided in ~\cite{CV} a construction of smooth solutions to the stationary Euler equations that are not beltramisable.

A very different picture appears if we require the solutions to be analytic, as in Arnold's original set-up. In this case Cieliebak and Volkov show that all nonvanishing steady Euler flows are indeed Beltrami fields for some metric. This implies, in view of the above mentioned Hutchings--Taubes' theorem, the following result (\cite{CV}, Corollary 3), which is also the main theorem that we prove in this paper:

\begin{theorem}\label{main1}
Let $M$ be a closed real analytic
3-manifold equipped with a metric $g$ and a volume form $\mu$ (which we do not assume to be the one induced by $g$). Let $X$ be a nonvanishing vector field $X$ satisfying the stationary Euler equations,
\[
X\times \curl X=\nabla B\,, \;\; \Div X=0\,.
\]
Suppose that $X$, $g$, $\mu$ and $B$ are analytic. Then, if $M$ is not a torus bundle over the circle,
$X$ has a periodic orbit.
\end{theorem}

We recall that given a metric $g$ and a volume form $\mu$ (not necessarily the one induced by $g$), the divergence, curl and vector product operators are defined by
\[
i_{\curl X} \mu= d(i_X g) \,,\,\, i_{X\times Y} g=i_{Y} i_X \mu \,,\,\, L_X \mu= \Div X \mu \,.
\]

The crucial step in the proof of Theorem \ref{main1} that we present is to show the following:

\begin{theorem}\label{main2}

Let $X$ be a nonvanishing vector field in a real analytic closed 3-manifold which is not a torus bundle. If $X$ has a non-trivial analytic first integral, then it has a periodic orbit.

\end{theorem}

The proof of Theorem \ref{main1} is then achieved by combining Theorem \ref{main2} with Taubes' solution to the Weinstein conjecture in dimension three \cite{TA} and Fuller's theorem \cite{FU} on the existence of periodic points of surface homemorphisms.

\subsection*{Comparison with previous results} As we mentioned above, neither Theorem \ref{main1} nor Theorem \ref{main2} are original. The strategy of the proof that we follow here isn't either: it already appears in the works of Etnyre and Ghrist \cite{EG1}, Rechtman \cite{RE2} and Cieliebak and Volkov \cite{CV}; the difference lies in the details of its implementation. 

More precisely, Etnyre and Ghrist proved in \cite{EG1} that, in the 3-sphere, a nonvanishing analytic Euler field always has a periodic orbit, and in fact the only thing preventing their arguments from yielding Theorem \ref{main1} was that the Weinstein conjecture in dimension three had not been solved yet. Rechtman's arguments in \cite{RE} also provide a proof of Theorem \ref{main1} along the same lines as the one presented here. Finally, as we have mentioned already, Theorem \ref{main1} is a corollary (\cite{CV}, Corollary 3) of Cieliebak and Volkov's beltramization result, and another proof following the same strategy as we do is presented in Remark 2.2 of their paper.

With regards to Theorem \ref{main2}, there are several results in the same spirit originally discovered by Fomenko in his work on Hamiltonian systems with Morse--Bott first integrals (\cite{FO}, Chapter 2); besides, it follows directly from the arguments in the proofs of Theorem 3.5 in \cite{EG1}, of Lemma 4.2 in \cite{RE} and of Remark 2.2 in \cite{CV}. The key in all of these arguments is to use the stratified geometry of the critical level sets of the first integral to show that, in the absence of periodic orbits, the manifold is foliated by tori; we proceed in a similar manner. Where our approach differs considerably is in the proof of the fact that $M$ does not just admit a foliation by tori, but is a torus bundle over the circle.

The paper is organized as follows: firstly we prove Theorem \ref{main2} in Section \ref{pmain2}, then Theorem \ref{main1} is proven in Section \ref{pmain1}.

\section{Proof of the Theorem \ref{main2}}\label{pmain2}

Let us denote by $h$ the analytic non-trivial first integral of the vector field $X$. We recall that being a first integral of $X$ means that $h$ is constant along the flow lines; in particular, since $h$ is $C^1$, we have $dh(X)=0$.

The idea of the proof is to first show that, if  $X$ has no periodic orbits, the level sets of $h$ define a foliation of $M$ where all leaves are tori; then we prove that this foliation yields actually a torus bundle over the circle.

\subsection{Step 1: the connected components of the level sets of $h$ are all tori}

Let $L$ be a connected component of the level set $h^{-1}(c)$, for some real number $c$ in the image of $h$.

If $c$ is a regular value, then $L$ is an oriented surface, which moreover is invariant under the flow of $X$; since $X$ has no zeroes, $L$ must be a torus.

If $c$ is a critical value, all we know in principle is that $L$ is a connected analytic subset of $M$. \L ojasiewicz proved \cite{LO} that these
sets have a controlled stratified structure. We begin by recalling some of their properties.

\begin{definition}
Let $L$ be an analytic set on a manifold $M$. We say that the rank of $L$
at a point $x\in L$ is $k$, $\text{rank}_L(x)=k$,  if there is an open neighborhood $U$
of $x$ in $M$ and analytic functions $f_{1},...f_{k}$ on $U$ such that
\[
U\cap L=\{f_{1}^{-1}(0)\cap...\cap f_{k}^{-1}(0),\,\text{ and } d_{x}f_{1},...d_{x}f_{k} \text{ are linearly independent}\}
\]

\end{definition}

Thus, if $\text{rank}_{L}(x)=k$, the set $L$ is a submanifold
of codimension $k$ in a neighborhood of $x$. 

We now define:
\[
L^{k}=\{x \in L,\,\text{rank}_L(x)=k\}
\]

The subset $L^{0}$ is  called the \emph{singular set} of $L$. We now summarize some standard geometric properties of analytic sets in connected real analytic manifolds that will be important for our purposes (see e.g \cite{BW}, Section 2 of \cite{LM}, or Theorem 6.3.3 in \cite{KP}):

\begin{enumerate}

\item The singular set $L^{0}$ is an analytic subset of $M$. If $L=L^{0}$,    then $L=M$; if $L \neq L^{0}$, then $L^0$ is a closed and nowhere dense
subset of $L$, and consists of the points at the intersection of the sets $\overline{L^{k}}$ with
$k>0$. 

\item For any point $x \in L^{0}$ we define rank$^{1}(x):=$ rank$_{L_0}(x)$, the rank of $L^0$ at the point $x$. We can then define analogously the sets $(L^{0})^{k}$; if $x$ is still on the singular set of $L^{0}$, i.e $x \in (L^{0})^{0}$, then we can define rank$^{2}(x)$
as the rank of $(L^{0})^{0}$ at $x$, and so on. This process eventually terminates, i.e, $(((L^{0})^{0})^{...})^{0}=\emptyset$.

\item If the dimension of the manifold is $n$ and $L \neq M$, we can define a finite filtration by closed subsets $L=L_{n-1}...\supset L_{0}$
(here $L_{k}\neq L^{k}$) such that for any $j$, $L_{j}\setminus L_{j-1}$
consists of the points $x$ with rank$_L(x)=n-j$ and of the points $x \in L^{0}$ with rank$^{i}(x)=n-j$ for some $i>0$. In other words, $L_{j}\setminus L_{j-1}$ is either a submanifold of dimension
$j$ or it is empty. We call the connected components of $L_j \setminus L_{j-1}$ the $j$-dimensional strata of $L$.

\end{enumerate}

Take now $L$ to be the connected component of the critical level set $h^{-1}(c)$; our first claim is that the strata $L_{k}\setminus L_{k-1}$ of dimension $k$ are invariant under the flow of $X$.

Indeed, consider the subsets $L^{k} \subset L$ of points with rank$_L= k$. For any $k>0$, it is clear from the definition of rank and the invariance of $L$ that $\phi^{t}_{X}(L^{k}) \subset L^{k}$ (here by $\phi^{t}_X$ we denote the flow of $X$). 

Now the singular set $L^{0}$ is clearly invariant since it is the complement in $L$ of the union of the sets $L^{k}$. Then again, by the definition of rank and the invariance of $L^{0}$, the submanifolds $(L^{0})^{k}$ consisting of the points in $L^{0}$ with rank$^{1}(x)=k$ are preserved by the flow. 

Arguing repeatedly as above, we conclude that the set of points $x$ with rank$^{j}(x)=k$, for any $j$ and $k$, is invariant by the flow. Since $L_{k}\setminus L_{k-1}$ is a union of such sets, it is also invariant.

We next proceed to analyze the geometry of the strata $L_{k} \setminus L_{k-1}$. First, observe that $L$ cannot have 3-dimensional strata, because then $h$ would be a constant, which is excluded. There cannot be 0-dimensional strata either, since by invariance these would be zeroes of $X$. The absence of 0-dimensional strata implies that any 1-dimensional stratum must be a closed curve, i.e a periodic orbit of $X$, which is also excluded. 

So there can only be 2-dimensional strata: $L$ must be a surface embedded in $M$, with the vector field $X$ tangent to it. Since $L$ is a critical level set, we cannot ensure its orientability: from the absence of zeroes of $X$ we can only conclude that $L$ is either a torus or a Klein bottle. However, a vector field tangent to a Klein bottle always has a periodic orbit \cite{GO}, so $L$ must be indeed a torus.

\subsection{Step 2: the foliation defined by the level sets of $h$ is locally trivial}

To finish our proof of Theorem \ref{main2} we need the following:

\begin{proposition}\label{levelsets2}
Let $L$ be a connected component of the level set $h^{-1}(c)$, for any $c$ in the image of $h$. There is an open set $U \subset M$ containing $L$ and a diffeomorphism $\Phi: U \rightarrow (-\delta, \delta) \times \TT^2$, for some positive number $\delta$, such that $h \circ \Phi^{-1}(\{t\} \times \TT^2)= \{c+t\}$. 
\end{proposition}

Indeed, if that is the case, the leaves of the foliation by tori defined by $h$ are separable. Since $M$ is closed, the leaf space of the foliation is a 1-dimensional, compact closed Hausdorff space, i.e a circle. By Ehresmann fibration theorem, the surjective submersion from $M$ to the leaf space $\SS^1$ is a locally trivial fibration, so $M$ is torus bundle over the circle and Theorem \ref{main2} is proven.

The rest of the Section is dedicated to the proof of the above Proposition. 

\begin{proof}[Proof of Proposition \ref{levelsets2}]

Choose any connected leaf $L$ of the foliation defined by the level sets of $h$.

Suppose first that the leaf $L$ is a connected component of a level set
$h^{-1}(c)$ with $c$ a regular value. Then the proposition follows trivially: in a sufficiently
small neighborhood $U$ of $L$, the gradient of $h$ does not vanish. We can thus
define the vector field $S:=\frac{\nabla h}{dh(\nabla h)}$ whose flow $\phi_{S}^{t}$ verifies:
\[ 
\frac{ \partial }{ \partial t}    \, h\big(\phi_{S}^{t}(x)\big)\big\rvert_{t=s}=dh(S)\big(\phi_{S}^{s}(x)\big)=1,\,\,\,\,h\big(\phi_{S}^{0}(x)\big)=c
\]

for any point $x\in L$. So we have $h(\phi_{S}^{t}(x))=h(x)+t$. Moreover, any point $y \in U$ can be uniquely written as $y=\phi_{S}^{t}(x)$ for some $t \in (-\delta, \delta)$ and $x \in L$ (maybe at the expense of taking a slightly smaller open set $U$), so we define the diffeomorphism $\Phi: U \rightarrow (-\delta, \delta) \times \TT^{2}$ simply as $\Phi(\phi_{S}^{t}(x))=(t, x)$ for $x \in L$.

Now consider a leaf $L$ which is a connected component of a level set
$h^{-1}(c)$ with $c$ a critical value. Without loss of generality, we can assume that $c=0$. The key in this case is to prove:

\begin{lemma}\label{local}

On a small enough tubular neighborhood $N_{\delta}(L)$
of $L$ we can write $h=g^{k}\cdot u$, with $k$ a positive integer, $g$ an analytic function
such that $\nabla g\neq0$, and $u$ a smooth
function with $u\neq0$.
\end{lemma}

Indeed, assume the above Lemma for the moment: then, on a neighborhood of $L$, the function $q:=g\cdot u^{\frac{1}{k}}$
has the same level sets as $h$ and, moreover, $\nabla q \neq 0$ on a neighborhood of $L$. Arguing exactly as we did in the case of non-critical level sets with $q$ in the role of $h$ we get a locally
trivial foliation near $L$. This concludes the proof of Proposition \ref{levelsets2} \end{proof}

\begin{proof}[Proof of Lemma \ref{local}]

For any point $x \in L$, we define the vanishing order of $h$ at $x$ to be the smallest integer $n$ such that the $n$-jet of the function $h$ at $x$, $J^{n}_{x}(h)$, does not vanish. Of course if $x$ is a critical point, $n$ is greater than 1 ($h(x)=0$ and $\nabla h (x)=0$).

We begin by establishing the following fact:

\begin{lemma}\label{vanishingorder}
Let $L$ be a connected component of $h^{-1}(0)$ with $0$ a critical value. If $X$ does not vanish and has no periodic orbits, all points in $L$ are critical points of $h$ of the same vanishing order. 
\end{lemma}

\begin{proof}[Proof of Lemma \ref{vanishingorder}]
Set
\[
K:=\underset{x \in L}{\text{ min }} \,\{n \in \NN \,\text{ such that\,}J_{x}^{n}(h)\neq 0\}\,,
\]
and define the following two sets partitioning $L$: 
\[
S_{0}:=\{x \in L,\, J_{x}^{K}(h)\neq0 \}
\]
\[
S_{1}:=\{x\,\in\,L,\,J_{x}^{K}(h)=0\} 
\,,
\]

that is, the points in $S_{0}$ are those with minimal vanishing
order, and $S_{1}=L\setminus S_{0}$. 

Notice that $S_{1}$ is an analytic subset of $L$, since it is defined as the zero set of the $k$-jet of $h$; and it cannot be the whole $L$ because by definition $S_0$ is not empty. Thus the codimension of $S_1$ in $M$ must be at least 2,
that is $S_{1}$ must consist of points and curves. 

But it is easy to check that all $n$-jets of $h$ are preserved
by the flow of $X$ (to see this, consider the matter in a flow box of the vector field), so that $S_{1}$
is invariant and thus it consists of zeroes or periodic orbits of $X$. Since this is excluded by our assumptions, we conclude that $S_1$ is empty and $L=S_{0}$, i.e, $h$ has the same vanishing order $K$ for all $x\in L$. \end{proof}

Now, the leaf $L$ has trivial normal bundle, so there
is a smooth function $F$ in $M$ such that $F^{-1}(0)=L$
and $\nabla F |_{L}\neq0$: indeed, we can simply define the function $F(x,\,t):=t$ for coordinates $(x,\,t)$ on a trivialization of a  tubular neighborhood $N_{\delta}(L)\simeq L\times (-\delta, \delta)$; and extend $F$ smoothly to the whole manifold. Observe that although $L=F^{-1}(0)$ is an analytic submanifold, we cannot ensure that $F$ is analytic.

By Whitney's
approximation theorem, we can find analytic functions arbitrarily
close to $F$ in any $C^{k}$ norm. Consider thus an analytic function $G$ very close to $F$ in the $C^{1}$ norm. Thom's isotopy theorem \cite{AR} ensures that, provided we take $G$ close enough, the level set $L'=G^{-1}(0)$
is diffeomorphic to $L$ and, moreover, that there is a \emph{smooth} family of isotopic embeddings
\[
e_{t}:\mathbb{T}^{2}\rightarrow M
\]
with $e_{0}(\TT^2)=L$ and $e_{1}(\TT^2)=L'$.

But since both $L$ and $L'$ are analytic, we can take the family of isotopic
embeddings $e_t$ to be analytic as well
(in the sense that the images are analytic submanifolds and the dependence
on $t$ is analytic) by a theorem due to Royden \cite{ROYD} (which precisely states that
given a smooth isotopy of embeddings between two analytic submanifolds we can find an analytic isotopy
of embeddings).

We define now a time-dependent vector field associated with the isotopy $e_t$:
\[
Y_s(e_s(q)):=\frac{\partial e_{t}}{\partial t} \Big\rvert_{t=s}(q)
\]
for $q \in \TT^2$. Note that for each $s$, $Y_s$ is only defined on the
embedded tori $e_s(\TT^2) \subset M$. It is nonetheless straightforward to extend it smoothly to a time-dependent vector field $Y_t$ on the whole $M$. In doing so, $Y_t$ cannot be ensured to be analytic anymore, but we can keep the extension analytic inside a small enough neighborhood of the original domain of definition. As we are going to see, this is enough for our purposes.

The non-autonomous flow $\phi^{t}_{Y_t}$ on $M$ satisfies, for any $p \in L$,
\[
\phi^{t}_{Y_t}(p)=e_t(e^{-1}_0(p))
\]
so that
\[
\phi_{Y_1}^{1}(L)=L'
\,
\]

and furthermore, $\phi_{Y_1}^{1}: M\rightarrow M$ is analytic on a small enough tubular neighborhood of $L$ that we denote by $N_{\delta}(L)$.

Define now the function
\[ 
g:=G\circ\phi_{Y}^{1}|_{N_{\delta}(L)} : N_{\delta}(L) \rightarrow \RR
\]

which is analytic as well, since $G$ and $\phi_{Y}^{1}|_{N_{\delta}(L)}$ are. 

We have that $L=g^{-1}(0)$
and $dg|_{N_{\delta}(L)}\neq0$. Now, since by Lemma \ref{vanishingorder} the function $h$ has the same vanishing order at all points $x \in L$, the Weierstrass Preparation Theorem (see e.g \cite{KP}, Theorem 6.3.1) readily implies that, on a neighborhood $V_x$ of any point $x \in L$, we can write $h=g^{K}\cdot u_{V_x}$, with $u_{V_x}$ a nonvanishing analytic function on $V_x$.

We take now a partition of unity $(\rho_{i},\,V_{i})$ of $N_{\delta}(L)$,
such that on each $V_{i}$ we have $h=g^{K}\cdot u_{i}$. Defining
\[
u=\underset{i}{\sum}\rho_{i}u_{i}
\]
we finally have
\[
h=\underset{i}{\sum}\rho_{i}g^{k}u_{i}=g^{k}\underset{i}{\sum}\rho_{i}u_{i}=g^{k}\cdot u\,
\]

Because of our use of a partition of unity, $u$ is not analytic, but it is clearly nonvanishing on $N_{\delta}(L)$. This finishes the proof of Lemma \ref{local}.\end{proof}

\section{Proof of Theorem \ref{main1}}\label{pmain1}

Any nonvanishing solution of the Euler equations is of one of the following four types:

\begin{enumerate}
\item $\curl X=0$.
\item $\curl X= \lambda X$ with some constant $\lambda \neq 0$.
\item $\curl X=fX$ with $f$ not constant.
\item $X \times \curl X=\nabla B$ with $\nabla B$ not identically zero. 
\end{enumerate}

We will prove Theorem \ref{main1} case by case. We note that the only cases where the analyticity assumptions will be needed are cases (iii) and (iv). 
\subsection{Case (i)} We begin by recalling a result due to Tischler \cite{TI}:
 
\begin{theorem}[Tischler \cite{TI}]\label{Tis}
Let M be
a closed manifold. If M has a nowhere vanishing closed 1-form
$\alpha$, then for any $\epsilon>0$ there is a surjective submersion $p: M \rightarrow \SS^1$ and an integer $n$ such that the 1-form $p^{*}d\theta$ (where $d\theta$ denotes the standard  length form on $\mathbb{S}^{1}$) satisfies $|n\alpha(Y(x))-p^{*}d\theta(Y(x))|_{C^0(M)} \leq \epsilon n|\alpha(Y(x))|$, for any vector field $Y$ at any point $x$.
\end{theorem}

Since $\curl X=0$ and $X$ has no zeroes, the 1-form $\lambda=i_{X}g$
is closed and nowhere vanishing. By Tischler's theorem, $M$
fibers over $\mathbb{S}^{1}$ and since $M$ is compact, the fibration $p$ is
proper, so it defines a fiber bundle. 

Note that by taking $\epsilon$ in the statement of Theorem \ref{Tis} smaller than 1 we ensure that $i_{X}\ p^{*} d \theta>0$, so the vector field $X$ is transverse to the fibers $p^{-1}(\theta)$.

We want to prove that, if $X$ does not have periodic
orbits, the fibers are tori. Take a fiber $\Sigma_{\theta}:=p^{-1}(\theta)$. Since $p^{*}d\theta(X)>0$, all the flow lines come back
to the fiber. We can then define the Poincar\'e recursion map of $X$ in $\Sigma_{\theta}$, which is a diffeomorphism $\psi_{X}:\Sigma_{\theta}\rightarrow\Sigma_{\theta}$ taking any point in the fiber to the first point in its forward trajectory to belong to the fiber again.

A periodic orbit of $X$ corresponds either to a fixed point
of the Poincar\'e recursion map $\psi_{X}$ or, more generally, to a \emph{periodic
point} of it, that is, a fixed point of the iterated
map $\psi_{X}^{k}=\psi_{X}\circ...\circ\psi_{X}$ for some integer
$k>0$.

So the proof of Theorem \ref{main1} in this case reduces to showing that if $\psi_{X}$ has no periodic
points, $\Sigma_{\theta}$ is a torus. This follows from a classical result of Fuller \cite{FU}, but in any case we will present a proof here. 

Define the Lefschetz zeta function of $\psi_X$ as
\[
Z_{\psi_X}(s):= \text{exp }\bigg(\sum_{k=1}^{\infty} \frac{\Lambda(\psi^{k}_{X})}{k} s^{k}\bigg)
\]

where $\Lambda(\psi_{X}^{k})$ denotes the Lefschetz number of $\psi^{k}_{X}$, 
\[
\Lambda(\psi_{X}^{k})=
\text{Tr}\big((\psi^{k}_{X})_{*0}\big)-\text{Tr}\big((\psi^{k}_{X})_{*1}\big)+\text{Tr}\big((\psi^{k}_{X})_{*2}\big)\,,
\]
and $(\psi^{k}_{X})_{*i}: H_{i}(\Sigma_{\theta}, \QQ) \rightarrow H_{i}(\Sigma_{\theta}, \QQ)$ is the linear map induced by $\psi^{k}_{X}$ on the $i$-th homology group of the fiber $\Sigma_{\theta}$.

Observe that by the Lefschetz fixed point theorem, if $\Lambda(\psi^{k}_{X}) \neq 0$ then $\psi_{X}$ has a periodic point. So if the map $\psi_{X}$ has no periodic points, all Lefschetz numbers are zero and we must have 
\[
Z_{\psi_X}(s)=1 \,.
\]
We now claim that the above identity is only possible if $\Sigma_{\theta}$ is a torus. Indeed, the Lefschetz zeta function can be rewritten as
\[
Z_{\psi_X}(s)=\frac{\text{exp}\bigg(-\sum_{k=1}^{\infty} \frac{s^{k}}{k} \text{ Tr}\big((\psi^{k}_{X})_{*1}\big) \bigg)}{\text{exp}\bigg(-\sum_{k=1}^{\infty} \frac{s^{k}}{k} \text{ Tr}\big((\psi^{k}_{X})_{*0}\big) \bigg)\text{exp}\bigg(-\sum_{k=1}^{\infty} \frac{s^{k}}{k} \text{ Tr}\big((\psi^{k}_{X})_{*2}\big) \bigg)}\,.
\]
Recall that for a square matrix $A$ one has
\[
\text{det}(I+A)=\sum_{n=0}^{\infty} \frac{1}{n!}\bigg(-\sum_{k=1}^{\infty} \frac{(-1)^{k}}{k} \text{Tr} A^{k}\bigg)^{n}
\]
where by $I$ we denote the identity matrix. Thus, we have
\[
\text{exp}\bigg(-\sum_{k=1}^{\infty} \frac{s^{k}}{k} \text{ Tr }\big((\psi^{k}_{X})_{*i}\big)\bigg)=\text{det}(I-s (\psi^{k}_{X})_{*i})
\]
so the Lefschetz zeta function reads
\begin{equation}\label{zeta}
Z_{\psi_X}(s)=\frac{\text{det}\big(I-s (\psi_{X})_{*1}\big)}{\text{det}\big(I-s (\psi_{X})_{*0}\big)\text{ det}\big(I-s (\psi_{X})_{*2}\big)}\,.
\end{equation}
In principle, $\text{det}(I-s (\psi_{X})_{*i})$ is a polynomial of degree $n_i \leq b_i$, where by $b_i$ we denote the $i$-th Betti number of $\Sigma_{\theta}$, $b_i:=\text{dim }(H_i(\Sigma_{\theta}, \QQ))$. But in fact, we have that the degree is exactly equal to the Betti number. To see this, note that
\[
\text{det}\big(I-s (\psi_{X})_{*i}\big)=s^{b_i} \text{det}\bigg(\frac{1}{s}I-(\psi_{X})_{*i}\bigg)=s^{b_i}p\bigg(\frac{1}{s}\bigg)
\]
where $p(\lambda)$ denotes the characteristic polynomial of $(\psi_{X})_{*i}$, that is, the polynomial whose zeroes are the eigenvalues of $(\psi_{X})_{*i}$. If there is no term of degree $b_i$ in $\text{det }(I-s (\psi_{X})_{*i})$, then there is no term of degree zero in $p(\lambda)$, and thus $\lambda=0$ is an eigenvalue of $(\psi_{X})_{*i}$. But this is impossible, because $\psi_{X}$ is a diffeomorphism, so $(\psi_{X})_{*i}$ is an isomorphism. So $n_i=b_i$.

Now, we have that $b_0=b_2=1$ and $b_1=2g$, where $g$ is the genus of $\Sigma_{\theta}$. Thus, the denominator in Eq. \eqref{zeta} is a polynomial of degree 2, while the numerator has degree $2g$. So for $Z_{\psi_{X}}(s)$ to be identically $1$ we must have $g=1$, i.e, $\Sigma_{\theta}$ must be a torus.

\subsection{Case (ii)} If $X$ satisfies the equation $\curl X= \lambda X$, it is immediate to see that $\alpha=i_X g$ is a contact form and 
\[
X':= \frac{X}{\alpha(X)}
\]
is the corresponding Reeb vector field of $\alpha$. Taubes' solution of the Weinstein conjecture \cite{TA} ensures that $X'$ has a periodic orbit, hence so does $X$. 

\subsection{Cases (iii) and (iv)} These two cases can be analyzed simultaneously,
as they have a key feature in common: the field $X$ has a non-trivial first integral, $f$ in case (iii) and $B$ in case (iv).

By the assumptions in Theorem \ref{main1}, $f$ and $B$ are  analytic. We can then apply Theorem \ref{main2} to conclude that if $X$ has no periodic orbits, $M$ is a torus bundle over the circle.

\section*{Acknowledgements}

The author thanks Daniel Peralta-Salas for his many valuable suggestions and advice. He is also grateful to the Max Planck Institute for Mathematics for its financial support and hospitality.


\begin{thebibliography}{99}\frenchspacing


\bibitem{AR} R. Abraham, J. Robbin, 
\emph{Transversal mappings
and flows}. Benjamin, New York, 1967.

\bibitem{AKh}

V.I. Arnold, B.A. Khesin, \emph{Topological Methods in Hydrodynamics}. Springer-Verlag, New York, 1998.


\bibitem{BW}

 
F. Bruhat, H. Whitney, 
Quelques propri\'et\'es fondamentales des ensembles analytiques-r\'eels. Comment. Math. Helv. 33 (1959) 132--160. 

\bibitem{CV} K. Cieliebak and E. Volkov,
A note on the stationary Euler equations of hydrodynamics. Ergodic Theory Dynam. Systems 37 (2017) 454--480. 

\bibitem{DH}

T. Dombre,  U. Frisch, J. M. Greene, M. H\'enon, A. Mehr, A. M. Soward, Chaotic streamlines in the ABC flows. J. Fluid Mech. 167 (1986) 353--391.

\bibitem{EP}
A. Enciso, D. Peralta-Salas, Existence of knotted vortex tubes in
steady Euler flows. Acta Math. 214 (2015) 61--134.


\bibitem{EPT}
A. Enciso, D. Peralta-Salas, F. Torres de Lizaur, Knotted structures in high-energy Beltrami fields on the torus and the sphere. Ann. Sci. \'Ec. Norm. Sup. 50 (2017) 4, 995--1016.

\bibitem{EG1} 

J. Etnyre, R. Ghrist, Contact topology and hydrodynamics I. Beltrami fields and the Seifert conjecture. Nonlinearity 13 (2000) 441--458.


\bibitem{FO}

 
A. T. Fomenko,
\emph{Integrability and nonintegrability in geometry and mechanics}. 
Translated from the Russian by M. V. Tsaplina. Mathematics and its Applications 31, Kluwer Academic Publishers Group, Dordrecht, 1988.

\bibitem{GO}

C. Godbillon, Syst\`emes dynamiques sur les surfaces. Publications IRMA Strasbourg, S\'eries de Math\'ematiques Pures et Appliqu\'ees 66 (1979).

\bibitem{H}
M. H\'enon, Sur la topologie des lignes de courant dans un cas particulier. C. R. Acad. Sci. Paris 262 (1966) 312--314.

\bibitem{HT}
M. L. Hutchings, C. H. Taubes,
The Weinstein conjecture for stable Hamiltonian structures. 
Geom. Topol. 13 (2009) 901--941.

\bibitem{Kh}
 
B. A. Khesin, Topological fluid dynamics. Notices Amer. Math. Soc. 52 (2005) 9--19.

\bibitem{KP}


S. G. Krantz, H. R. Parks, \emph{A primer of real analytic functions}. 
Birkh\"auser, Boston, 2002.

\bibitem{LO}

S. \L ojasiewicz, Ensembles semi-analytiques, Inst. Hautes \'Etudes Sci., Bures-sur-Yvette, 1964.


\bibitem{LM}

D. B. Massey, D. T L\^e, 
Notes on real and complex analytic and semianalytic singularities. \emph{Singularities in geometry and topology}, 81--126, World Sci. Publ., Hackensack, NJ, 2007.



\bibitem{FU} 

F.B. Fuller, The existence of periodic points. Ann. of Math. 57 (1953) 229--230.


\bibitem{P}

D. Peralta-Salas, Selected topics on the topology of ideal fluid flows. Int. J. Geom. Meth. Mod. Phys. 13 (2016) 1--23.

\bibitem{RE2}

A. Rechtman, Use and disuse of plugs in foliations. PhD Thesis, ENS Lyon, 2009.
%
\bibitem{RE}

A. Rechtman, Existence of periodic orbits for geodesible vector fields on closed 3-manifolds. Ergodic Theory Dynam. Systems 30  (2010) 1817--1841.

\bibitem{ROYD}

H. L. Royden,
The analytic approximation of differentiable mappings.
Math. Ann. 139 (1960)  171--179. 


\bibitem{TA}

C. H. Taubes, The Seiberg-Witten equations and the Weinstein conjecture. Geom. Topol. 11 (2007) 2117--2202.

\bibitem{TI}

D. Tischler, On bering certain foliated manifolds over $\mathbb{S}^{1}$. Topology 9 (1970) 153--154.

\end{thebibliography}
\end{document}